\documentclass[a4paper]{amsart}
\usepackage{amsfonts}
\usepackage{amssymb}
\usepackage{nicefrac}
\usepackage{enumerate}
\usepackage{url,bm}
\usepackage{color}      
\usepackage{capt-of}
\usepackage{graphicx}

\allowdisplaybreaks

\newcommand{\R}{\mathbb{R}}
\newcommand{\N}{\mathbb{N}}
\newcommand{\C}{\mathbb{C}}
\newcommand{\Z}{\mathbb{Z}}

\newcommand{\T}{\mathcal{T}}
\newcommand{\X}{\mathcal{X}}
\newcommand{\A}{\mathcal{A}}

\newcommand{\W}{\Omega}
\newcommand{\bPhi}{\bm{\varPhi}}
\newcommand{\bPsi}{\bm{\varPsi}}

\newcommand{\w}{\omega}
\newcommand{\g}{\gamma}

\newcommand{\s}{\bm{\mathcal{F}}}
\newtheorem{theorem}{Theorem}[section]
\newtheorem{lemma}[theorem]{Lemma}

\newtheorem{proposition}[theorem]{Proposition}

\def\clspan{{\overline{\mathrm{span}}}}
\def\rk{\text{rk}}

\newcommand{\esssup}{\mathop{\rm ess\,sup}}

\theoremstyle{remark}
\newtheorem{remark}[theorem]{Remark}

\theoremstyle{definition}
\newtheorem{definition}[theorem]{Definition}


\title[Linear combinations of generators of MI spaces]{Linear combinations of generators in multiplicatively invariant spaces}

\author{V. Paternostro}

\begin{document}

\address{\textrm{(V. Paternostro)}
Departamento de Matem\'atica, Facultad de Ciencias Exactas y Naturales, Universidad  de
Buenos Aires, Ciudad Universitaria, Pabell\'on I, 1428 Buenos Aires, Argentina and IMAS-CONICET, Consejo 
Nacional de Investigaciones Cient\'ificas y T\'ecnicas, Argentina}
\email{vpater@dm.uba.ar}

\maketitle

\begin{abstract}
Multiplicatively invariant   (MI) spaces are closed subspaces of $L^2(\W,\mathcal{H})$ that are invariant under multiplications of (some) functions in $L^{\infty}(\W)$; they were first introduced by Bownik and Ross in 2014. In this paper we work with MI spaces  that are finitely generated. We prove that almost every  set of functions constructed by taking linear combinations of  the generators of a  finitely generated MI space is a new set of generators for the same space and we give necessary and sufficient conditions  on the linear combinations to preserve frame properties.
We then apply our results on MI spaces  to systems of translates in the context of locally compact abelian groups and we extend some results  previously proven for systems of integer translates in $L^2(\R^d)$.
\end{abstract}

\section{Introduction}
Given a vector valued space $L^2(\W,\mathcal{H})$ where $\W$ is a $\sigma$-finite measure space and $\mathcal{H}$ is a separable Hilbert space and  given $D$ a determining set for $L^1(\W)$ (see Section \ref{sec-MI} for a precise definition) a {\it multiplicatively invariant} (MI) space is a closed subspace of $L^2(\W,\mathcal{H})$ that is invariant under multiplications by  functions in $D$.
A particular case of MI spaces are the well-known {\it doubly invariant spaces} introduced by Helson \cite{Hel64} and Srinivasan \cite{Sri64}. Recently, MI spaces were introduced as presented here in  \cite{BR14} where also they were characterized in terms of range functions. The reason why MI spaces appear on the scene  is because they are strongly connected  to {\it shift invariant} (SI) spaces. In the classical euclidean case, a SI space is a closed subspace in $L^2(\R^d)$  that is invariant under translations by integers. These type of spaces are typically considered in sampling theory, \cite{AG01, Sun10, Sun05, ZZC06} and they also play a fundamental role in approximation theory as well as in frame and wavelet theory \cite{Gro04,HW96,Mal89}.  Shift invariant spaces have proven to be very useful models in many problems in signal and image processing. Due to their importance in theory and applications, their structure  has been deeply analyzed during the last twenty five years \cite{Bow00, dBDVR94a, dBDVR94b,  Hel64, RS95}.

Every SI space can be generated by a set $\Phi$ of functions in $L^2(\R^d)$ in the sense that it is
the closure of the space spanned by the integer translations of the functions in $\Phi$. When $\Phi$ is a finite set, we say that the SI space is finitely generated.
Concerning finitely generated SI spaces, a particular problem  of interest for us, is the following:
suppose that $\Phi=\{\phi_1, \ldots, \phi_m\}$ generates the SI space $V$, that is, $V=\clspan\{T_k\phi_j:\, k\in\Z, j=1, \ldots, m\}$. For $\ell\leq m$, let $\Psi=\{\psi_1, \ldots, \psi_\ell\}$ be a set of functions constructed by taking linear combinations of the functions in $\Phi$, i.e. $\psi_i=\sum_{j=1}^ma_{ij}\phi_j$ for $1\leq i\leq \ell$. The question is which are the linear combinations that produce new sets of generators for $V$?  and if in addition we know that $\{T_k\phi_j\}_{k\in\Z, j=1,\ldots, m}$ is a frame for $V$, when is also $\{T_k\psi_i\}_{k\in\Z, i=1,\ldots, \ell}$ a frame for $V$?
These two questions were completely answered in \cite{BK06} and \cite{CMP14}.
The problem of plain generators was addressed in \cite{BK06} where the authors proved  that almost every   linear combination of the original generators of $V$ generates $V$.
Regarding the second question, in \cite{CMP14}, the authors exactly characterized those  linear combinations that transfer the frame property from $\{T_k\phi_j\}_{k\in\Z, j=1,\ldots, m}$ to $\{T_k\psi_i\}_{k\in\Z, i=1,\ldots, \ell}$.

In the present work we study the questions formulated above but for MI spaces.
In our main result we show that almost every linear combination of generators of a MI space produces a new set of generators for the same space. We also characterize those linear combinations that preserve uniform frames
(see Definition \ref{def-uniform-frames}). Our results are then in the spirit of those in \cite{BK06, CMP14}.
As a first step, we work with finite dimensional subspaces. We prove that given a finite set of vectors $\mathcal{V}$ in a Hilbert space, almost every finite set of vectors constructed by taking linear combinations of the vectors in $\mathcal{V}$ spans the same subspace that $\mathcal{V}$ spans. This result will be the core  of what we then prove for MI spaces and  we also believe  it is of interest by itself.

As a consequence, we obtain similar results to \cite{BK06, CMP14} but for SI spaces considered in more general contexts than $L^2(\R^d)$.
The theory of shift invariant spaces has been extended to the setting of locally compact abelian (LCA)  groups, mainly in two different directions. First, in \cite{ BR14, CP10, KR10} SI spaces are subspaces of $L^2(G)$ where $G$ is an LCA group and the translations are taken along a subgroup $H$ of $G$ such that $G/H$ is compact. The  case when $H$ is  discrete was addressed in \cite{CP10, KR10} and in the recent paper \cite{BR14} the authors worked with the non-discrete case. Second, one can consider SI spaces in $L^2(\X)$ where $\X$ is a measure space and the translations are defined by the action of a discrete LCA group on $\X$, \cite{BHP14b}. In both cases, SI spaces  were characterized in terms of range functions using fiberizations techniques  obtaining results that extend those proven for SI spaces  in $L^2(\R^d)$ in \cite{Bow00}. This last fact is what connects SI spaces with  MI spaces. Then, our results in MI spaces allow us to provide a unified treatment to the problem of when linear combinations
of
generators in system of translates preserve generators and frame generators in both contexts described above.

The paper is organized as follows. In Section \ref{sec-generators} we show that generators of finite dimensional subspaces are generally preserved by the action of taking linear combinations.
Section \ref{sec-results-MI} is devoted to MI spaces. We first summarize in Section \ref{sec-MI} the basic properties of MI spaces. We prove in Section \ref{sec-MI-generators} that almost every linear combination of  generators of a MI space yields to a new set of generators for the same space (Theorem \ref{lc-of-MI-generators}). In Section \ref{sec-preserving-uniform-frame} we address the problem of preserving uniform
frames. Finally in Section \ref{sec-applications} we apply the result we obtain for MI spaces to systems of translates.

We finish this introduction by stating the notation we use.
\subsection{Notation and Definitions}
Here we set the notation we will use in the next sections, we recall the definition of frames and some basic results about linear algebra that will be important in what follows.

\begin{definition}\label{def-frame}
Let $\mathcal{H}$ be a separable Hilbert  space and   $\{f_k\}_{k\in \Z}$ be a sequence in $\mathcal{H}.$
The sequence $\{f_k\}_{k\in \Z}$ is said to be  a {\it frame} for $\mathcal{H}$ if there exist $0<\alpha\leq \beta$ such that
\begin{equation*}\label{eq-frame}
\alpha\|f\|^2\leq \sum_{k\in \Z} |\langle f,f_k\rangle|^2\leq \beta\|f\|^2
\end{equation*}
for all $f\in\mathcal{H}$. The constants $\alpha$ and $\beta $ are called {\it frame bounds}.
\end{definition}

For a set of vectors $\mathcal{X}=\{x_1, \ldots, x_n\}\subseteq \mathcal{H}$ we denote by $S(\mathcal{X})$ the subspace spanned by $\mathcal{X}$, i.e. $S(\mathcal{X})=\textnormal{span}{\{x_1, \ldots, x_n\}}$. The {\it Gramian} associated to $\mathcal{X}$ is the matrix $G_{\mathcal{X}}$  in $\C^{n\times n}$ whose entries are
$(G_{\mathcal{X}})_{ij}=\langle x_i, x_j\rangle.$
The Gramian is a positive-semidefinite matrix satisfying $G_{\mathcal{X}}^*=G_{\mathcal{X}}$.

Denote by $K_{\mathcal{X}}:\C^n\to\mathcal{H}$ the {\it synthesis operator} associated to $\mathcal{X}$ given by $K_{\mathcal{X}}c=\sum_{j=1}^nc_jx_j$ and by $K_{\mathcal{X}}^*:\mathcal{H}\to\C^n$ its adjoint, the so-called {\it analysis operator},  given by $K_{\mathcal{X}}^*h=\{\langle h, x_j\rangle \}_{j=1}^n$.
Note that the matrix representation of the operator $K_{\mathcal{X}}^*K_{\mathcal{X}}$ in the canonical basis of $\C^n$  is the transpose of $G_{\X}$, $G_{\X}^t$. It follows then that,
\begin{align}\label{dim-gramian}
\rk(G_{\mathcal{X}})&=\rk(G_{\mathcal{X}}^t)=\dim(Im(K_{\mathcal{X}}^*K_{\mathcal{X}}))\nonumber\\
&=\dim(Im(K_{\mathcal{X}}K_{\mathcal{X}}^*))
=\dim(Im(K_{\mathcal{X}}))\\
&=\dim(S(\mathcal{X}))\nonumber.
\end{align}

The set $\mathcal{X}$ is always a frame for $S(\mathcal{X})$ and its  frame bounds  are related to the Gramian in the following way: $0<\alpha\leq\beta$ are frame bounds of $\mathcal{X}$ if and only if $\Sigma(G_{\mathcal{X}})\subseteq \{0\}\cup [\alpha, \beta]$, where $\Sigma(G_{\mathcal{X}})$ is the set of eigenvalues of $G_{\mathcal{X}}$.

If $E\subseteq \C^d$ we indicate by $|E|$ its Lebesgue measure.

\section{Linear combinations preserving generators of subspaces}\label{sec-generators}
In this section, we are interested in
studying  which are the linear combinations of the generators of a finite dimensional subspace that preserve the property of generating the same subspace. Let us explain the problem in detail.
Let $\mathcal{H}$ be a separable Hilbert space and consider a finite set of elements in $\mathcal{H}$, $\mathcal{V}=\{v_1, \ldots, v_m\}$. Denote by $V$  the vector
whose entries are the elements of $\mathcal{V}$, i.e. $V=(v_1, \ldots, v_m)$.
For any $\ell$ such that $r\leq \ell\leq m$, where $r=\dim(S(\mathcal{V}))$, let $\mathcal{W}=\{w_1, \ldots, w_{\ell}\}$ be a set
constructed by taking linear combinations of elements of $\mathcal{V}$. This is, for  $i=1,\ldots,\ell$,  $w_i=\sum_{j=1}^m a_{ij}v_j$ for   some complex scalars $ a_{ij}$.
Collecting the coefficients of the linear combinations  in a matrix $A=\{a_{ij}\}_{i,j}\in \C^{\ell\times m}$, we can write in  matrix  notation
\begin{equation}\label{link-V-W}
 W=AV^t,
\end{equation}
where  $W=(w_1, \ldots, w_{\ell})$.
Therefore, the question is which are
the matrices $A$ that transfer  the property of being a set of generators for $S(\mathcal{V})$ from $\mathcal{V}$ to $\mathcal{W}$.
We shall answer this by showing  that for almost every matrix $A\in \C^{\ell\times m}$, the set $\mathcal{W}$ spans $S(\mathcal{V})$.

\begin{theorem}\label{thm-ppal}
 Let $\mathcal{V}=\{v_1, \ldots, v_m\}\subseteq \mathcal{H}$ and let $r=\dim(S(\mathcal{V}))$. For any $\ell$ such that $r\leq \ell\leq m$, consider
 the set of matrices
$\mathcal{R}=\{A\in\C^{\ell\times m}\,:\, S(\mathcal{V})=S(\mathcal{W}) \}$ where $\mathcal{W}$ is obtained from $\mathcal{V}$ by the relationship $W=AV^t$. Then,
$\C^{\ell\times m}\setminus \mathcal{R}$ has zero Lebesgue measure.
\end{theorem}

\vspace{-0.2cm}
In order to prove  Theorem \ref{thm-ppal} we first give a description of the set $\mathcal{R}$  in terms of the Gramians associated to $\mathcal{V}$ and $\mathcal{W}$ when $\mathcal{V}$ and $\mathcal{W}$ are linked by  \eqref{link-V-W}.
The connection  between the Gramians $G_{\mathcal{V}}$ and $G_{\mathcal{W}}$ is provided in the upcoming lemma (see also \cite[Proposition 2.5] {CMP14}).\vspace{-0.2cm}

\begin{lemma}\label{lemma-gramian}
Let $\mathcal{V}=\{v_1, \ldots, v_m\}\subseteq \mathcal{H}$. If $\mathcal{W}=\{w_1, \ldots, w_{\ell}\}$ is constructed from $\mathcal{V}$ by taking linear combinations of its elements as in \eqref{link-V-W} then, the Gramians associated to $\mathcal{V}$ and $\mathcal{W}$ satisfy
$G_{\mathcal{W}}=AG_{\mathcal{V}}A^*.$\vspace{-0.2cm}
\end{lemma}
\begin{proof}
Since  $W=AV^t$, we obtain
\begin{align*}
(G_{\mathcal{W}})_{ij}
&=\big\langle \sum_{k=1}^ma_{ik}v_k, \sum_{r=1}^ma_{jr}v_r\big\rangle
=\sum_{r, k=1}^ma_{ik}\overline{a_{jr}}\underbrace{\langle v_k, v_r\rangle}_{(G_{\mathcal{V}})_{kr}}
=(AG_{\mathcal{V}}A^*)_{ij}.
\end{align*}
\end{proof}

For $\mathcal{V}$ and $\mathcal{W}$ as in Theorem \ref{thm-ppal}, i.e linked by \eqref{link-V-W},
we have that
$S(\mathcal{W})\subseteq S(\mathcal{V})$. Hence, $S(\mathcal{W})= S(\mathcal{V})$ if and only if $\dim(S(\mathcal{W}))= \dim(S(\mathcal{V}))$. Now, by \eqref{dim-gramian} and Lemma \ref{lemma-gramian},
$\dim(S(\mathcal{W}))= \dim(S(\mathcal{V}))$ is true if and only if $\rk(AG_{\mathcal{V}}A^*)=\rk(G_{\mathcal{V}})$.
As a consequence, the set $\mathcal{R}$ in Theorem \ref{thm-ppal} can be described as the set of matrices  preserving the rank of $G_{\mathcal{V}}$ under the action $AG_{\mathcal{V}}A^*$
\begin{equation}\label{eq:R}
\mathcal{R}=\{A\in\C^{\ell\times m}\,:\, \rk(AG_{\mathcal{V}}A^*)=\rk(G_{\mathcal{V}})\}.
\end{equation}
Having the description of $\mathcal{R}$ in terms of the Gramian of $\mathcal{V}$, the  proof of Theorem \ref{thm-ppal} follows from the next  rank-preserving result:

\begin{proposition}\label{prop-rank-preserving}
Let $G$ be a  positive-semidefinite matrix in $\C^{m\times m}$ such  that  $G=G^*$ and  let $r=\rk(G)$. For any  $r\leq \ell\le m$
define  the set $\mathcal{R}(G)=\{A\in\C^{\ell\times m}\,:\, \rk(G)=\rk(AGA^*)\}$.
Then, $\mathcal{N}(G):=\C^{\ell\times m}\setminus \mathcal{R}(G)$ has zero Lebesgue measure.
\end{proposition}

\begin{proof}
 Since $G$ is a self-adjoint  positive-semidefinite matrix, there exists a unitary matrix $U\in\C^{m\times m}$ and positive scalars $\lambda_1\geq\ldots \geq \lambda_r >0$ such that $U^*GU=D$ where $D$ is the diagonal matrix in $\C^{m\times m}$, $D=diag(\lambda_1, \ldots,  \lambda_r, 0, \ldots, 0)$. In particular, $r=\rk(G)=\rk(D)$.

 Note that for any  $A\in\C^{\ell\times m}$, $A$ preserves the rank of $G$ under the action $AGA^*$ if and only if $AU$ preserves the rank of $D$ under the action $AUD(AU)^*$. Therefore,
 \begin{align*}
 \mathcal{N}(G)U&=\{AU\,:\, A\in\C^{\ell\times m}, \,\rk(G)\neq\rk(AGA^*)\}\\
 &=\{B\in\C^{\ell\times m}, \rk(D)\neq\rk(BDB^*)\}=
 \mathcal{N}(D).
 \end{align*}
 Since $U$ is a unitary matrix, the mapping $A\mapsto AU$ from $\C^{\ell\times m}$ in itself preserves Lebesgue measure, implying  $|\mathcal{N}(G)|=|\mathcal{N}(D)|$. Thus, it is enough to show that $|\mathcal{N}(D)|=0$.

Let $B\in\C^{\ell\times m}$ be written by column-blocks  as $B=(B_1|B_2)$ where the columns of $B_1$ are the first $r$ columns of $B$ and the columns of $B_2$ are the last $m-r$ columns of $B$.
Then,
$$\rk(BDB^*)=\rk(BD^{1/2}(BD^{1/2})^*)=\rk(BD^{1/2})=\rk(B_1).$$
Thus, $\mathcal{N}(D)=\{B=(B_1|B_2)\in\C^{\ell\times m}\,:\, B_1\in \C^{\ell\times r}, B_2\in \C^{\ell\times m-r}, \rk(B_1)<r\}$. Since the set of matrices in $\C^{\ell\times r}$ which are not full rank has zero Lebesgue measure, the result follows.
\end{proof}

\section{Linear combination of generators of MI spaces}\label{sec-results-MI}
In the previous section we showed that almost all linear combinations of generators of a finite dimensional subspace produce a new set of vectors spanning the same subspace.
We want to study now a similar problem  but in the context of multiplicatively invariant (MI) spaces of $L^2(\Omega, \mathcal{H})$.  The concept of MI spaces was recently introduced in the general setting of $L^2(\Omega, \mathcal{H})$ in \cite{BR14} as a generalization of the very well-known doubly invariant spaces proposed by Helson in \cite{Hel64} and Srinivasan in \cite{Sri64} for $\Omega=\mathbb{T}$.
We shall prove that an analogous result to Theorem \ref{thm-ppal} can be obtained for MI spaces in   $L^2(\Omega, \mathcal{H})$.
The main difference here lies in the meaning  of the word ``generator'' which,  for MI spaces,  differs from the notion of generator for a subspace.
To properly describe and state the result we shall prove in this case, we first summarize the basic properties of MI spaces in Section \ref{sec-MI}.

\subsection{Multiplicatively invariant spaces in $L^2(\Omega, \mathcal{H})$}\label{sec-MI}
The  material we collect here is a summary of the content of \cite[Section 2]{BR14}. See \cite{BR14} for details and proofs.

Let $(\Omega, \mu)$ be a $\sigma$-finite measure space and let $\mathcal{H}$ be a separable Hilbert space.
The vector valued space $L^2(\Omega, \mathcal{H})$ is the space of measurable functions $\Phi:\Omega\to\mathcal{H}$ such that $\|\Phi\|^2=\int_{\Omega}\|\Phi(\omega)\|^2_{\mathcal{H}}\,d\mu(\omega)<+\infty$. The inner product in $L^2(\Omega, \mathcal{H})$ is given by $\langle \Phi, \Psi\rangle=\int_{\Omega}\langle \Phi(\omega), \Psi(\omega)\rangle_{\mathcal{H}}\,d\mu(\omega)$.

For defining MI spaces in $L^2(\Omega, \mathcal{H})$ we require the concept of  determining set for $L^1(\Omega)$. A set $D\subseteq L^{\infty}(\Omega)$ is said to be a {\it determining set} for $L^1(\Omega)$ if
for every $f\in L^1(\Omega)$ such that $\int_{\Omega}f(\omega)g(\omega)\,d\mu(\omega)=0\,\,\forall g\in D$, one has $f=0$.
In the setting of Helson \cite{Hel64}, a determining set is the set of exponentials with integer parameter, $D=\{e^{2\pi i k\cdot}\}_{k\in\Z}\subseteq L^{\infty}(\mathbb{T})$.

\begin{definition}\label{def-MI}
A closed subspace $M\subseteq  L^2(\Omega, \mathcal{H})$ is {\it multiplicatively  invariant} with respect to the determining set $D$ for $L^1(\Omega)$ (MI space for short) if
$$\Phi\in M\Longrightarrow g\Phi\in M, \,\, \textrm{ for any }\,\,g\in D.$$
For an at most countable (meaning finite or countable)  subset $\bm{\varPhi}\subseteq L^2(\Omega, \mathcal{H})$  define
$
M_D(\bm{\varPhi})= \overline{\mbox{span}}\{g\Phi\colon \Phi\in\bm{\varPhi}, g\in D\}.
$
The subspace $M_D(\bm{\varPhi})$ is called the multiplicatively  invariant space  generated by $\bm{\varPhi}$, and we say that $\bPhi$ is a set of generator for $M_D(\bPhi)$. When $\bm{\varPhi}$ is  finite, $M=M_D(\bm{\varPhi})$ is said to be finitely generated by $\bm{\varPhi}$. In that case, we define the {\it length} of $M$ as
$$\ell(M)= \min\{n\in\N\colon \exists \,\,\Phi_1, \cdots,\Phi_n \in
M \textrm{ with } M=M_D(\Phi_1, \ldots,\Phi_n)\}.$$
\end{definition}

One of the most important properties of MI spaces is their characterization in terms of measurable range functions.
A {\it range function} is a mapping $J:\Omega\to \{\textrm{closed subspaces of }\mathcal{H}\}$ equipped
with the orthogonal projections $P_J(\omega)$ of $\mathcal{H}$ onto $J(\omega)$. A range functions is said to be {\it measurable} if for every $a,b\in\mathcal{H}$, $\omega\mapsto\langle P_J(\omega)a,b\rangle$ is measurable as a function from $\Omega$ to $\C$.

\begin{theorem}\cite[Theorem 2.4]{BR14}\label{thm-MI-rango}
Suppose that $L^2(\Omega)$ is separable, so that $L^2(\Omega, \mathcal{H})$ is also separable.
Let $M$ be a closed subspace of $L^2(\Omega, \mathcal{H})$ and $D$ a determining set for $L^1(\Omega)$.
Then, $M$ is an MI space with respect to $D$ if and only if there exists a measurable range function $J$ such that
$$M=\{\Phi\in  L^2(\Omega, \mathcal{H})\,:\, \Phi(\omega)\in J(\omega) \textrm{ a.e. } \omega\in\Omega\}.$$
Identifying range functions that are equal almost everywhere, the correspondence between MI spaces and measurable range functions is one-to-one and onto.

Moreover, when $M=M_D(\bm{\varPhi})$ for some at most countable set $\bm{\varPhi}\subseteq  L^2(\Omega, \mathcal{H})$ the range function associated to $M$ is
$$J(\w)=\clspan\{\Phi(\w):\Phi\in\bm{\varPhi}\}, \quad \textrm{a.e. }\w\in\W.$$
\end{theorem}

\subsection{Linear combinations of MI-generators.}\label{sec-MI-generators}
We can now properly state what we want to prove.
Fix $D\subseteq L^{\infty}(\Omega)$ a determining set for $L^1(\Omega)$. Let us suppose that $M$ is a finitely generated MI space with respect to $D$. This is, $M=M_D(\bm{\varPhi})$ where $\bm{\varPhi}=\{\Phi_1, \ldots, \Phi_m\}\subseteq
L^2(\Omega, \mathcal{H})$. For a number $\ell$ such that $\ell(M)\leq\ell\leq m$ we construct a new set of functions of $M$, $\bm{\varPsi}=\{\Psi_1, \ldots, \Psi_{\ell}\}$ say,  by taking linear combinations of $\{\Phi_1, \ldots, \Phi_m\}$
as we did  for the case of generators for a finite dimensional Hilbert spaces in Section \ref{sec-generators}.
More precisely, for each $1\leq i\leq \ell$, $\Psi_i=\sum_{j=1}^ma_{ij}\Phi_j$, and collecting the coefficients in a matrix $A\in\C^{\ell\times m}$, we write $\varPsi=A\varPhi^t$ where $\varPhi=(\Phi_1, \ldots, \Phi_m)$ and $\varPsi=(\Psi_1, \ldots, \Psi_{\ell})$. The question is now, which are the matrices $A\in\C^{\ell\times m}$ that transfer the property of being a generator set for $M$ from $\bm{\varPhi}$ to
$\bm{\varPsi}$. 

\begin{theorem}\label{lc-of-MI-generators}
Let $M$ be a finitely generated MI space  and $\bm{\varPhi}=\{\Phi_1,\cdots, \Phi_m\}\subseteq L^2(\Omega, \mathcal{H})$ be such that $M=M_D(\bm{\varPhi})$ where $\ell(M) \leq m$.
For $\ell(M)\leq\ell\leq m$, consider the set of matrices
$\mathcal{R}=\{A\in\C^{\ell\times m}\,:\, M=M_D(\bm{\varPsi}), \varPsi=A\varPhi^t \}.$ Then,
$\C^{\ell\times m}\setminus \mathcal{R}$ has zero Lebesgue measure.
\end{theorem}

Observe that this result is analogous to the one we proved for the case of generators for subspaces, Theorem \ref{thm-ppal}. As we mentioned before, the generator set $\bPhi$ generates $M_D(\bPhi)$ as a MI space. This fact  changes the nature of the problem and as a consequence, the proof of Theorem \ref{lc-of-MI-generators} requires more subtle techniques than those used for proving Theorem \ref{thm-ppal}.

For the proof of the above theorem we need the following known result.
\begin{lemma}\label{fubini}
 Let $(X, \mu)$ and $(Y, \nu)$ be measure spaces and  $F\subseteq X\times Y$ be a measurable set. The sections of $F$ are
 $F_x=\{y\in Y\,: (x,y)\in F\}$ and $F_y=\{x\in X\,: (x,y)\in F\}$. Then, $(\mu\times\nu)(F)=0$ if and only if
 $\nu(F_x)=0$ for $\mu$-a.e. $x\in X$ if and only if
 $\mu(F_y)=0$ for $\nu$-a.e. $y\in Y$.
\end{lemma}

\begin{proof}[Proof of Theorem \ref{lc-of-MI-generators}]
Along all this proof the relationship between $\bm{\varPhi}$ and $\bm{\varPsi}$ will be always  $\varPsi=A\varPhi^t$ for some matrix $A\in \C^{\ell\times m}$ so, we  will not repeat this again. We denote by $J_{\bm{\varPhi}}$ and $J_{\bm{\varPsi}}$ the measurable range functions associated to $M_D(\bm{\varPhi})$ and $M_D(\bm{\varPsi})$ respectively and for each $\w\in\W$, ${\bm{\varPhi}}(\w)=\{\Phi_1(\w), \ldots, \Phi_m(\w)\}$ and ${\bm{\varPsi}}(\w)=\{\Psi_1(\w), \ldots, \Psi_{\ell}(\w)\}$. Note that since $\varPsi=A\varPhi^t$, $\varPsi(\w)=A\varPhi(\w)^t$, where $\varPhi(\w)=(\Phi_1(\w), \ldots, \Phi_m(\w))$ and
$\varPsi(\w)=(\Psi_1(\w), \ldots, \Psi_{\ell}(\w))$. We now proceed as in \cite{BK06}.
By Theorem \ref{thm-MI-rango} and and the reasoning we used to obtain \eqref{eq:R}, we deduce that
\begin{align*}
 \mathcal{R}&=\{A\in\C^{\ell\times m}\,:\, M_D(\bm{\varPhi})=M_D(\bm{\varPsi}) \}\\
 &=\{A\in\C^{\ell\times m}\,:\, J_{\bm{\varPhi}}(\w)=J_{\bm{\varPsi}}(\w), \textrm{ a.e. } \w\in\W\}\\
 &=\{A\in\C^{\ell\times m}\,:\, S({\bm{\varPhi}}(\w))=S({\bm{\varPsi}}(\w))\textrm{ a.e. } \w\in\W\}\\
 &=\{A\in\C^{\ell\times m}\,:\, \rk(G_{\bm{\varPhi}(\w)})=\rk(AG_{\bm{\varPhi}(\w)}A^*)\textrm{ a.e. } \w\in\W\},
\end{align*}
where in the last equality  $G_{\bm{\varPhi}(\w)}$ is the Gramian associated to ${\bm{\varPhi}(\w)}$.
Since for a.e. $ \w\in\W$, $\rk(G_{\bm{\varPhi}(\w)})\geq\rk(AG_{\bm{\varPhi}(\w)}A^*)$, we then want to prove  that the set
\begin{equation}\label{null-set}
\{A\in\C^{\ell\times m}\,:\, \rk(G_{\bm{\varPhi}(\w)})>\rk(AG_{\bm{\varPhi}(\w)}A^*){\tiny\textrm{ for } \w\textrm{ belonging to a set of positive measure}}\}
\end{equation}
has  zero Lebesgue measure.

Let $\mathcal{F}=\{(\w, A)\in\W\times\C^{\ell\times m}\,:\, \rk(G_{\bm{\varPhi}(\w)})>\rk(AG_{\bm{\varPhi}(\w)}A^*)\}.$
Since for each $1\leq i\leq m$, $\Phi_i$ is a measurable function, so are the entries of $G_{\bm{\varPhi}(\w)}$.
On the other hand, the rank of any matrix is the largest of the absolute values of its minors.  Thus, since the determinant is a polynomial on the entries of the matrix, it follows that the rank of a matrix with measurable entries is a measurable function. Now, since $\mathcal{F}=f^{-1}((0, +\infty))$ where $f$ is the measurable function $f(\w, A)=\rk(G_{\bm{\varPhi}(\w)})-\rk(AG_{\bm{\varPhi}(\w)}A^*)$, it turns out that $\mathcal{F}$ is a measurable subset of $\W\times\C^{\ell\times m}$.

The sections of the $\mathcal{F}$ are denoted by $\mathcal{F}_{\w}$ and $\mathcal{F}_{A}$. By Proposition \ref{prop-rank-preserving} we know that $|\mathcal{F}_{\w}|=0$ for a.e $\w\in\W$ and hence, by Lemma \ref{fubini}
$\mu(\mathcal{F}_{A})=0$ for a.e. $A\in\C^{\ell\times m}$.
Note that the set given in \eqref{null-set} is exactly
$\{A\in\C^{\ell\times m}\,:\, \mu(\mathcal{F}_{A})>0\}$. Therefore,  it has zero Lebesgue measure.
\end{proof}

\subsection{Linear combinations preserving uniform frames.}\label{sec-preserving-uniform-frame}
As we mentioned in the introduction, we want to give a unified treatment for the problem of when linear combinations preserve generators and frame generators in systems of translates, where the ``systems of translates'' are considered   in different contexts. This is why we work at the level of the vector valued functions. For addressing  the frame case, we need to introduce the following  definition which, at this point,  may seem a bit  artificial. However, we shall see that it has complete sense in each of the different contexts we want to consider.

\begin{definition}\label{def-uniform-frames}
 Let $\bm{\varPhi}\subseteq L^2(\Omega, \mathcal{H})$ be an at most countable set  and let $J$ be the  measurable range function defined as $J(\w)=\clspan\{\Phi(\w):\Phi\in\bm{\varPhi}\}, \quad \textrm{a.e. }\w\in\W$. We say that $\bm{\varPhi}$ is a {\it uniform frame} for $J$ if there exist constant $0<\alpha\leq \beta$ such that, for a.e. $\w\in\W$,  the set $\{\Phi(\w):\,\Phi\in\bPhi\}$ is a frame for $J(\w)$ with frame bounds $\alpha$ and $\beta$.
\end{definition}

Fix $D\subseteq L^{\infty}(\W)$ a determining set for $L^1(\W)$ and  suppose that $\bPhi$ is a finite set of functions in $L^2(\Omega, \mathcal{H})$ such that it is a uniform frame for $J$ where $J$ is the measurable range function associated to $M=M_D(\bm{\varPhi})$.
Then, Theorem \ref{lc-of-MI-generators} tells us that almost every linear combination of the functions in $\bPhi$ produces a new set of generator $\bPsi$ of $M$. In particular, this is saying us that for a.e. $\w\in\W$, $\bPsi(\w)$ is a new set of generators for $J(\w)$. Thus, we are interested in knowing  which are the linear combinations that also preserve uniform frames.  This is, if $A$ is such that $\varPsi=A\varPhi^t$, what is the property $A$ must satisfy so that $\bPsi$ is a uniform frame for $J$? We shall answer this question by completely characterizing matrices $A$ that preserve uniform frames it terms of angles between subspaces.
To this end, we first recall the notion of Friedrichs angle, \cite{Deu95, HJ95, Kat76}.

Let $S, T\neq \{0\}$ be subspaces of $\C^n$. The {\it Friedrichs  angle between $S$ and $T$}
is the angle in $[0, \frac{\pi}{2}]$ whose cosine is defined by
$$
{\bm{\mathcal{G}}}[S,T]=\sup\{|\langle x,y\rangle|:\, x\in S\cap (S\cap T)^{\perp},\, \|x\|=1, \, y\in T\cap (S\cap T)^{\perp},\, \|y\|=1\}.
$$
We define ${\bm{\mathcal{G}}}[S,T]=0$ if $S=\{0\}$, $T=\{0\}$, $S\subseteq T$ or $T\subseteq S$.
As usual, the sine of the Friedrichs angle is defined as $\s\,[S,T]=\sqrt{1-{\bm{\mathcal{G}}}[S,T]^2}.$

We can now state the characterization of matrices that preserve uniform frames.

\begin{theorem}\label{thm-frames}
Let $ \bPhi=\{\Phi_1,\ldots, \Phi_m\}\subseteq L^2(\Omega, \mathcal{H})$ be a uniform frame for $J$ where $J$ is the measurable range function associated to $M=M_D(\bPhi)$ and suppose that $\ell(M)\leq\ell\leq m$.
Let $A\in\C^{\ell\times m}$ be a matrix and consider $\bPsi=\{\Psi_1,\ldots, \Psi_{\ell}\}$ where $\varPsi=A\varPhi^t$. Then, $\bPsi$ is a uniform
frame for $J$ if and only if $A$ satisfies the following two conditions
\begin{enumerate}
\item $A\in\mathcal{R}$ where $\mathcal{R}$ is as in Theorem \ref{lc-of-MI-generators}.
\item There exists $\delta>0$ such that $\s\,[Ker(A), Im(G_{\bPhi(\w)})]\ge \delta$ for a.e. $\w\in\W$.
\end{enumerate}
\end{theorem}

The proof of Theorem \ref{thm-frames} is based in the fact that $\bPhi$ is a uniform frame with frame bounds $\alpha$ and $\beta$ for $J$ if and only if $\Sigma(G_{\bPhi(\w)})\subseteq [\alpha, \beta]\cup\{0\}$ for a.e. $\w\in\W$. Therefore, the task is to prove that conditions
$(1)$ and $(2)$ of Theorem \ref{thm-frames} guarantee that the positive eigenvalues of $AG_{\bm{\varPhi}(\w)}A^*$ are uniformly bounded. This can be done using \cite[Proposition 3.3]{CMP14}, which is an adaptation of a result on singular values of composition of operator of Antezana et al. \cite{ACRS05}.
Having at hand these results, the complete proof of Theorem \ref{thm-frames} is a  readily adaptation of the proof of Theorem 4.4  in \cite{CMP14}.
For the convenience of the reader we provide it here.
\begin{proof}[Proof of Theorem \ref{thm-frames}]
 For a matrix $G$ such that it is positive-semidefinite and $G=G^*$ we denote by $\lambda_{-}(G)$ its  smallest non-zero eigenvalue. For  any matrix $B$ we denote by $\sigma(B)$ the smallest non-zero singular value of $B$.

 Let $0<\alpha\leq\beta$ be the frame bounds of $\bPhi$. Then,
since $\Sigma(G_{\bPhi(\w)})\subseteq [\alpha, \beta]\cup\{0\}$ for a.e. $\w\in\W$, we have that $\alpha\leq \lambda_{-}(G_{\bPhi(\w)})$ and $\|G_{\bPhi(\w)}\|\leq\beta$ for a.e. $\w\in\W$.

Suppose first that $\bPsi$ is a uniform frame for $J$ with frame bounds $0<\alpha'\leq\beta'$.
In particular, since correspondence between MI spaces and range functions is one-to-one and onto, $\bPsi$ is a generator set for $M=M_D(\bPhi)$ and then $A\in\mathcal{R}$. Thus, we are under hypotheses of \cite[Proposition 3.3]{CMP14} and so
\begin{align*}
\lambda_{-}(G_{\bPsi(\w)})= \lambda_{-}(AG_{\bPhi(\w)}A^*)&\le \|A\|^2\|G_{\bPhi(\w)}\| \,\s[Ker(A), Im(G_{\bPhi(\w)})]\\
&\leq\|A\|^2 \beta\,\s[Ker(A), Im(G_{\bPhi(\w)})].
\end{align*}
Thus, $\alpha'\leq\|A\|^2 \beta\,\s[Ker(A), Im(G_{\bPhi(\w)})]$ and condition $(2)$ follows with
$\delta=\frac{\alpha'}{\|A\|^2 \beta}$.

Suppose now that $(1)$ and $(2)$ are satisfied for some matrix $A$. Then, we apply again \cite[Proposition 3.3]{CMP14} to get
\begin{equation}\label{eq:eigenvalue}
\lambda_{-}(AG_{\bPhi(\w)}A^*)\ge \sigma(A)^2\lambda_{-}(G_{\bPhi(\w)})\,\s[Ker(A), Im(G_{\bPhi(\w)})]^2\ge \sigma(A)^2 \alpha \delta^2,
\end{equation} for a.e. $\w\in\W$.

Note that, $\|G_{\bPsi(\w)}\|= \|AG_{\bPhi(\w)}A^*\|\le \|A\|^2\|G_{\bPhi(\w)}\|\le \|A\|^2\beta$ and hence the
eigenvalues of $G_{\bPsi(\w)}$ are bounded above by $\|A\|^2\beta$.
Combining this fact together with \eqref{eq:eigenvalue} we obtain that $\Sigma(G_{\bPsi(\w)})\subseteq \big[ \,\sigma(A)^2 \alpha \delta^2, \|A\|^2\beta\,\big] \cup\{0\}$ for a.e. $\w\in\W$ and then $\bPsi$ is a uniform frame for $J$.
\end{proof}

When the new set of generators has exactly $\ell(M)$ elements the following theorem  can be shown. For its proof see \cite[Theorem 4.7]{CMP14}.

\begin{theorem}\label{thm-Radu}
Let $ \bPhi=\{\Phi_1,\ldots, \Phi_m\}\subseteq L^2(\Omega, \mathcal{H})$ be a uniform frame for $J$ where $J$ is the measurable range function associated to $M=M_D(\bPhi)$ and let $\ell(M)=\ell\leq m$.
Let $A\in\C^{\ell\times m}$ be a matrix and consider $\bPsi=\{\Psi_1,\ldots, \Psi_{\ell}\}$ where $\varPsi=A\varPhi^t$.
Then, $\bPsi$ is a uniform
frame for $J$ if and only if  $AA^*$ is invertible and
\begin{equation*}
 \esssup_{\w\in \W} \|(I_m-A^*(AA^*)^{-1}A)G_{\bPhi(\w)}G_{\bPhi(\w)}^{\dagger}\|<1.
\end{equation*}
Here, $I_m$ is the identity in $\C^{m\times m}$ and $G_{\bPhi(\w)}^{\dagger}$ is the Moore-Penrose pseudoinverse of
$G_{\bPhi(\w)}$.
\end{theorem}

\begin{remark}
 It might be the case that condition $(2)$ in Theorem \ref{thm-frames} is not satisfied for any matrix $A$. An example of this situation is given in \cite[Example 4.12]{CMP14} for the case of system of translates in $\R^d$ but it can be easily adapted to the setting of MI spaces. Indeed. In $L^2((-1/2, 1/2]^2, \ell^2(\Z^2))$ consider $\Phi_1$ and $\Phi_2$ the vector valued functions given by $\Phi_1(\w_1,\w_2)=-\sin(2\pi\w_1)e_o$ and  $\Phi_1(\w_1,\w_2)=e^{2\pi i\w_2}\cos(2\pi\w_1)e_o$
 where $e_0$ is the sequence in $\ell^2(\Z^2)$ that takes the value $1$ at $(0,0)$ and $0$ otherwise. As a determining set take $D=\{e^{2\pi i \langle (k,j), \cdot\rangle}\}_{(k,j)\in\Z^2}$. Then, for $M_D(\Phi_2, \Phi_2)$ there is no matrix satisfying condition $(2)$ in Theorem \ref{thm-frames}. See \cite[Example 4.12]{CMP14} for details.
\end{remark}

\section{Application to systems of translates}\label{sec-applications}
In this section we show how the previous results can be applied to systems of translates.
As we will see,
there exists a  connection between systems of translates and vector valued functions which of course depends on the context where the systems of translates are considered. The link is  what we call {\it fiberization isometry}.

\subsection{Systems of translates on LCA groups}
Here we work with systems of translates in the context of locally compact abelian groups. Given $G$ a second countable  LCA group written additively, we consider translates of functions in $L^2(G)$ along a subgroup $H\subseteq G$ such that $G/H$ is compact.
A closed subspace $V\subseteq L^2(G)$ is said to be {\it $H$-invariant} (or invariant under translations in $H$) if for every $f\in V$, $T_hf\in V$ for all $h\in H$ where $T_h$ denotes the translation by $h$, i.e $T_hf(x)=f(x-h)$.
Subspaces that are $H$-invariant  were characterized using range functions and fiberization techniques in \cite{CP10, KR10} when $H$ is discrete.
Recently in \cite{BR14}, a similar characterization was obtained only assuming that $G/H$ is compact (i.e $H$ not necessarily discrete). An important point to get these characterizations is to see the space $L^2(G)$
as a vector valued space of the form $L^2(\W, \mathcal{H})$ for some particular choices of $\W$ and $\mathcal{H}$. Now we briefly describe how to do this.

Let $\widehat{G}$ be the dual group of $G$, that is, the set of continuous characters on $G$. For $x\in G$ and $\g\in\widehat{G}$ we use the notation $(x,\g)$ for the complex value that $\g$ takes at $x$. For any subgroup $H\subseteq G$, $H^*$ the {\it annihilator} of $H$
is the subgroup of $\widehat{G}$, $H^*=\{\g\in\widehat{G}:\, (h,\g)=1, \forall\, h\in H\}$.
Let us assume from now on that $H$ is a {\it co-compact} subgroup of $G$, this is, $G/H$ is compact.
Then, by the duality theorem \cite[Lemma 2.1.3]{Rud62},  it follows that  $H^*$ is discrete.
Now fix $\W\subseteq \widehat{G}$ a measurable section of the quotient $\widehat{G}/H^*$ whose existence is a consequence of \cite[Lemma 1.1]{Mac52}.
When the Haar measures of the groups involved here are appropriately chosen,   the following result shows that $ L^2(G)$ is isometrically isomorphic to the vector valued space $L^2(\W, \ell^2(H^*))$.  For its proof see \cite[Proposition 3.3]{CP10} and \cite[Proposition 3.7]{BR14}.

\begin{proposition}\label{prop-fiber}
 The fiberization mapping $\T:L^2(G)\to L^2(\W, \ell^2(H^*))$ defined by
 $$\T f(\w)=\{\widehat{f}(\w+\delta)\}_{\delta\in H^*}$$
 is an isometric isomorphism and it satisfies
 $\T T_hf(\w)=(-h, w)\T f(\w)$ for all $f\in L^2(G)$ and all $h\in H$.
 Here, $\widehat{f}$ denotes the Fourier transform of $f$.
\end{proposition}

The fiberization isometry of Proposition \ref{prop-fiber} allows us to see $L^2(G)$ as the vector valued space
$ L^2(\W, \ell^2(H^*))$.
Under this isometry,  $H$-invariant spaces of $L^2(G)$ exactly correspond with MI spaces of $ L^2(\W, \ell^2(H^*))$. Let us explain this correspondence in detail.
The determining set behind the notion of MI spaces in $ L^2(\W, \ell^2(H^*))$ is the set of functions
$D=\{(h, \cdot)\chi_{\W}(\cdot)\}_{h\in H}$, \cite[Corollary 3.6]{BR14}.
Thus, by Proposition \ref{prop-fiber}, one has that $V\subseteq L^2(G)$ is a $H$-invariant space if and only if $M=\T V$ is a MI space with respect to $D$. Therefore, we can also identify $H$-invariant spaces with measurable range functions as it was shown in \cite[Theorem 3.10]{CP10} and \cite[Theorem 3.8]{BR14}:
\begin{theorem}\label{thm-characterization-LCA}
Let $V \subseteq  L^2(G)$  be a closed subspace and $\T$ the mapping  defined in Proposition \ref{prop-fiber}.
Then, $V$ is $H$-invariant if and only if there exists a measurable range function $J$ such that
$$V=\{f\in L^2(G): \T f(\w)\in J(\w) \textrm{ for a.e. } \w\in\W\}.$$
Once one identifies range functions that are equal almost everywhere, the correspondence between measurable range functions and $H$-invariant spaces is one-to-one and onto.
When $V=\clspan\{T_h\varphi: h\in H, \varphi\in\A\}$ for an at most countable set $\A\subseteq  L^2(G)$, the measurable range function associated to $V$ is given by
$$J(\w)=\clspan\{\T\varphi(\w):\, \varphi\in\A\}, \quad \textrm{a.e. } \w\in\W.$$
\end{theorem}

Frames of translates can be also characterized using range functions and the fiberization isometry.
Indeed, when $\A\subseteq  L^2(G)$ is a countable set,  frames of the form $\{T_h\varphi: h\in H, \varphi\in\A\}$ for $V=\clspan\{T_h\varphi: h\in H, \varphi\in\A\}$ correspond to uniform frames for $J$ where $J$ is the measurable range function associated with $V$. For the case when $H$ is discrete, this fact was proven in \cite[Theorem 4.1]{CP10}. When $H$ is not discrete but co-compact, the set $\{T_h\varphi: h\in H, \varphi\in\A\}$ is not indexed by a discrete set and then one needs to work with the notion of continuous frame (see  Definition 5.1 in \cite{BR14} for details).  The characterization of continuous frames in terms of range functions was given in \cite[Theorem 5.1]{BR14}.
In the upcoming theorem, we state the characterization of frames of translates using range functions without distinguishing  between the discrete and the continuous case. The reader must have in mind that
when $H$ is not discrete the word ``frame'' refers to the notion of continuous frame  as \cite[Definition 5.1]{BR14}.

\begin{theorem}\label{thm-uniform-frames-LCA}
Let  $\A\subseteq  L^2(G)$ be a countable set, let $J$ be the measurable range function associated to $V=\clspan\{T_h\varphi: h\in H, \varphi\in\A\}$ and let $\T$ be the mapping of Proposition \ref{prop-fiber}. Then, the following conditions are equivalent:
\begin{enumerate}
 \item $\{T_h\varphi: h\in H, \varphi\in\A\}$ is a frame for $V$ with frame bounds $0<\alpha\leq \beta$.
 \item $\{\T\varphi:\, \varphi\in\A\}$ is a uniform frame for $J$ with frame bounds $0<\alpha\leq \beta$. This is, for a.e. $\w\in\W$, $\{\T\varphi(\w):\, \varphi\in\A\}$ is a frame for $J(\w)$ with (uniform) frame bounds $0<\alpha\leq \beta$.
\end{enumerate}
\end{theorem}

We already have all the ingredients we need to see how the results
of Section \ref{sec-results-MI}  can be applied to this setting.
Fix $\{\phi_1, \ldots, \phi_m\}\subseteq L^2(G)$ and consider the $H$-invariant space generated by
$\{\phi_1, \ldots, \phi_m\}$, $V=\clspan\{T_h\phi_j: h\in H, 1\leq j\leq m\}$.
By taking linear combinations of $\{\phi_1, \ldots, \phi_m\}$ we want to construct new sets of generators for $V$. As we did before for the case of generator for subspaces and for MI spaces, we consider sets of functions in $L^2(G)$,  $\{\psi_1, \ldots, \psi_{\ell}\}$ where for every $1\leq j\leq m$, $\psi_j=\sum_{i=1}^m a_{ij}\phi_j$ and $\ell$ is a number between the length of the MI space $\T V$ and $m$.  Collecting the coefficients of the linear combinations in a matrix $A\in\C^{\ell\times m}$ and letting $\Phi$ and $\Psi$ be the vectors of functions $\Phi=(\phi_1, \ldots, \phi_m)$ and
$\Psi=(\psi_1, \ldots, \psi_{\ell})$ we can write $\Psi=A\Phi^t$.
In the next theorem we prove that for almost every matrix $A\in\C^{\ell\times m}$, the functions $\{\psi_1, \ldots, \psi_{\ell}\}$ generate $V$. This result extends \cite[Theorem 1]{BK06} to the context of LCA groups, and moreover, since $H$ is allowed to be non discrete, it is  new even in the case when $G=\R^d$.

\begin{theorem}\label{thm-translate-lca-gen}
 Given $\{\phi_1, \ldots, \phi_m\}\subseteq L^2(G)$ let $V=\clspan\{T_h\phi_j: h\in H, 1\leq j\leq m\}$
 and let $\ell(M)$ be the length of $M=\T V$ where $\T$ is the fiberization isometry of Proposition \ref{prop-fiber}.
 For $\ell(M)\leq \ell\leq m$, let $\mathcal{R}$ be the set of matrices $A=\{a_{ij}\}_{ij}\in\C^{\ell\times m}$ such that the linear combinations $\psi_j=\sum_{i=1}^m a_{ij}\phi_j$ generate $V$, i.e. $V=\clspan\{T_h\psi_i: h\in H, 1\leq i\leq \ell\}$. Then, $\C^{\ell\times m}\setminus \mathcal{R}$ has Lebesgue zero measure.
\end{theorem}

\begin{proof}
Let $A\in\C^{\ell\times m}$ and consider the functions $\{\psi_1, \ldots, \psi_{\ell}\}$ where $\Psi=A\Phi^t$. Then, $\{\psi_1, \ldots, \psi_{\ell}\}$ is a set of generators for $V$ if and only if $\{\T\psi_1, \ldots, \T\psi_{\ell}\}$ generates $M$ as a MI space with respect to
$D=\{(h, \cdot)\chi_{\W}(\cdot)\}_{h\in H}$.
Denoting  $\bPhi=\{\T\phi_1, \ldots, \T\phi_{m}\}$, $\bPsi=\{\T\psi_1, \ldots, \T\psi_{\ell}\}$,
$\varPsi=(\T\psi_1, \ldots, \T\psi_{\ell})$ and $\varPhi=(\T\phi_1, \ldots, \T\phi_{m})$, we that have that $\mathcal{R}=\{A\in\C^{\ell\times m}: M=M_D(\bPsi), \, \varPsi=A\varPhi^t\}$.
Thus, by Theorem \ref{lc-of-MI-generators},  $\C^{\ell\times m}\setminus \mathcal{R}$ has Lebesgue zero measure.
\end{proof}

 The next theorem is an extension to LCA groups of \cite[Theorem 4.4]{CMP14}.
\begin{theorem}\label{frames-LCA}
Let $\{\phi_1,\cdots, \phi_m\}\subseteq L^2(G)$ such that $\{T_h\phi_j: h\in H, 1\leq j\leq m\}$ is a frame for $V=\clspan \{T_h\varphi_j: h\in H, 1\leq j\leq m\}$ and suppose that $\ell(M)\leq \ell\leq m$, where $M=\T V$ and $\T$ is as in Proposition \ref{prop-fiber}.
Let $A\in\C^{\ell\times m}$ be a matrix and consider $\{\psi_1,\cdots, \psi_{\ell}\}$ where $\Psi=A\Phi^t$. Then, $\{T_h\psi_i: h\in H, 1\leq i\leq \ell\}$ is a
frame for $V$ if and only if $A$ satisfies the following two conditions
\begin{enumerate}
\item $A\in\mathcal{R}$, where $\mathcal{R}$ is as in Theorem \ref{thm-translate-lca-gen}.
\item There exists $\delta>0$ such that $\s\,[Ker(A), Im(G_{\bPhi(\w)})]\ge \delta$ for a.e. $\w\in\W$, where
$G_{\bPhi(\w)}$ is the Gramian associated to $\{\T\phi_1(\w), \ldots, \T\phi_{m}(\w)\}$.
\end{enumerate}
\end{theorem}

\begin{remark}
 Given $\Delta\subseteq \widehat{G}$  a co-compact subgroup of the dual group of $G$ and $\A\subseteq L^2(G)$, let us consider the system $\{M_\delta\phi:\,\phi\in\A, \delta\in\Delta\}$ where $M_\delta$ is the modulation operator given by $M_\delta\phi(x)=(x,\delta)\phi(x)$.
 Since under the Fourier transform modulations become translations, all the results we have proven for systems of translates can be reformulated for  systems of modulations.
 Furthermore, one may also consider systems of time-frequency translates  $\{M_\delta T_h\phi:\,\phi\in\A, \delta\in\Delta, h\in H\}$ where $H\subseteq G$ and $\Delta\subseteq \widehat{G}$ are discrete subgroups and $\A\subseteq L^2(G)$. Spaces that are the closure of the span of systems of time-frequency translates are called {\it shift-modulation invariant spaces} or Gabor spaces. Using fiberization techniques and range functions, a characterization of theses spaces was given in \cite{CP12}. Therefore, this setting is one more example where the results of Section \ref{sec-results-MI} can be applied.
\end{remark}

\subsection{Discrete LCA groups acting on $\sigma$-finite  measure spaces}
We are interested now in systems of functions constructed by the action of a discrete  LCA group $\Gamma$ on $L^2(\X)$ where $(\X, \mu)$ is a $\sigma$-finite measure space. We will work with  {\it quasi-$\Gamma$-invariant} actions.
This notion was  introduced in \cite{HSWW10} and then extended to the non abelian case in \cite{BHP14a}.
Fix $\Gamma$ a discrete countable LCA group. Let $(\X,\mu)$ be a $\sigma$-finite measure space and $\sigma:\Gamma\times\X\to\X$ a measurable action satisfying the following conditions:
\begin{enumerate}
 \item [(i)] for each $\g\in\Gamma$ the map $\sigma_{\g}:\X\to\X$ given by $\sigma_{\g}(x):=\sigma(\g,x)$ is $\mu$-
measurable;
\item [(ii)] $\sigma_{\g}(\sigma_{\g'}(x))=\sigma_{\g+\g'}(x)$, for all $\g,\g'\in\Gamma$ and for all $x\in\X$;
\item [(iii)] $\sigma_{e}(x)=x$ for all $x\in\X$, where $e$ is the identity of $\Gamma$.
\end{enumerate}
The action $\sigma$ is said to be {\it quasi-$\Gamma$-invariant} if there exists  a measurable function $J_{\sigma}:\Gamma\times\X\to\R^{+}$, called {\it Jacobian of $\sigma$}, such that
$d\mu(\sigma_{\g}(x))=J_{\sigma}(\g,x)d\mu.$
To each quasi-$\Gamma$-invariant action $\sigma$ we can associate a unitary representation $T_{\sigma}$ of $\Gamma$ on $L^2(\X)$ given by
$T_{\sigma}(\g)f(x)=J_{\sigma}(-\g,x)^{\frac1{2}}f(\sigma_{-\g}(x)).$

Given a quasi-$\Gamma$-invariant action $\sigma$, we say that a closed subspace $V$ of $L^2(\X)$ is  {\it $\Gamma$-invariant} if
$$f\in V\Longrightarrow T_{\sigma}(\g)f\in V, \,\, \textrm{ for any }\,\,\g\in\Gamma.$$
When $L^2(\X)$ is separable, each $\Gamma$-invariant spaces is of the form
$V=\clspan\{T_{\sigma}(\g)\varphi:\g\in\Gamma, \varphi\in\A\}$ for some at most countable set $\A\subseteq L^2(\X)$.

In order to obtain the analogous results to Theorems \ref{thm-translate-lca-gen} and \ref{frames-LCA} for systems of the form $\{T_{\sigma}(\g)\phi_j\}_{j=1}^m$ using the machinery of MI spaces of Section \ref{sec-results-MI} we first  need to establish a connection between $L^2(\X)$   and a vector valued space of the type $L^2(\W, \mathcal{H})$. We can do this assuming
that the quasi-$\Gamma$-invariant action $\sigma$ satisfies the {\it tiling property}. This is, there exists a measurable subset $C\subseteq \X$ such that $\mu(\X\setminus\bigcup_{\g\in\Gamma}\sigma_{\g}(C))=0$ and $\mu(\sigma_{\g}(C)\cap\sigma_{\g'}(C))=0$ whenever $\g\neq\g'$. In this case it can be shown (see \cite{BHP14a} and \cite{BHP14b}) that there exists  an isometric isomorphism between $L^2(\X)$ and the vector valued space $L^2(\widehat{\Gamma}, L^2(C))$.
\begin{proposition}\cite[Proposition 3.3]{BHP14b}\label{abelian-tau-isometry}
 The mapping $\T_{\sigma}: L^2(\X) \longrightarrow  L^2(\widehat{\Gamma}, L^2(C))$ defined by
 $$\T_{\sigma}[\psi](\alpha)(x):=\sum_{\g\in\Gamma}
 \left[(T_{\sigma}(\g)\psi)(x)\right](-\g, \alpha)$$
 is an isometric isomorphism and it satisfies
$\T_{\sigma}[T_{\sigma}(\g)\psi](\alpha)=(\g,\alpha)\T_{\sigma}[\psi]$.
\end{proposition}

As for the case of ordinary translates  of the previous section, the isomorphism $\T_{\sigma}$ of
Proposition \ref{abelian-tau-isometry} connects $\Gamma$-invariant spaces of $L^2(\X)$ with MI spaces in
$L^2(\widehat{\Gamma}, L^2(C))$. Here the  determining set $D$ is the set of characters of $\widehat{\Gamma}$. More precisely, for every $\g\in\Gamma$, let $X_{\g}:\widehat{\Gamma}\to \C$ be the homomorphism defined as $X_{\g}(\alpha)=(\g,\alpha)$. Then, by the Pontrjagin Duality
\cite[Theorem 1.7.2]{Rud62}, $\{X_{\g}\}_{\g\in\Gamma}$ is the set of  characters of $\widehat{\Gamma}$
and thus, as a consequence of the uniqueness of the Fourier transform, $D= \{X_{\g}\}_{\g\in\Gamma}$ is a determining set for $L^1(\widehat{\Gamma})$. Therefore, it is possible to characterize $\Gamma$-invariant spaces using range functions obtaining a similar result to Theorem \ref{thm-characterization-LCA}. Furthermore, it can be also proven characterizations of frames of the form
$\{T_{\sigma}(\g)\varphi:\g\in\Gamma, \varphi\in\A\}$ for $V=\clspan\{T_{\sigma}(\g)\varphi:\g\in\Gamma, \varphi\in\A\}$ in the same spirit of Theorem  \ref{thm-uniform-frames-LCA}. We do not include here the complete statements of these results because we consider it is clear for the reader how they must be (see \cite[Theorems 4.3 and 5.1]{BHP14b} for details and proofs).

In a similar way as we proved Theorem \ref{thm-translate-lca-gen}, the following result can be shown:

\begin{theorem}
Given $\{\phi_1, \ldots, \phi_m\}\subseteq L^2(\X)$ let $V=\clspan\{T_{\sigma}(\g)\phi_j: \g\in \Gamma, 1\leq j\leq m\}$
 and let $\ell(M)$ be the length of $M=\T_{\sigma} [V]$ where $\T_{\sigma}$ is as in Proposition \ref{abelian-tau-isometry}. For $\ell(M)\leq \ell\leq m$,
let $\mathcal{R}$ be the set of matrices $A=\{a_{ij}\}_{ij}\in\C^{\ell\times m}$ such that the linear combinations $\psi_j=\sum_{i=1}^m a_{ij}\phi_j$ generate $V$, i.e. $V=\clspan\{T_{\sigma}(\g)\psi_i: \g\in \Gamma, 1\leq i\leq \ell\}$. Then, $\C^{\ell\times m}\setminus \mathcal{R}$ has Lebesgue zero measure.

 If in addition $\{T_{\sigma}(\g)\phi_j: \g\in \Gamma, 1\leq j\leq m\}$ is a frame for $V$, then $\{T_{\sigma}(\g)\psi_i: \g\in \Gamma, 1\leq i\leq \ell\}$ is also a frame for $V$ if and only if $A\in\mathcal{R}$ and there exists $\delta>0$ such that $\s\,[Ker(A), Im(G_{\bPhi(\alpha)})]\ge \delta$ for a.e. $\alpha\in\widehat{\Gamma}$, where
$G_{\bPhi(\alpha)}$ is the Gramian associated to $\{\T_{\sigma}[\phi_1](\alpha), \ldots, \T_{\sigma}[\phi_{m}](\alpha)\}$.
 \end{theorem}

 \subsection*{Acknowledgements}
This project was carried out when the author was  supported by a fellowship for postdoctoral researchers from the Alexander von Humboldt Foundation. She wants to thank Gitta Kutyniok for supporting her during her stay in Technische Universit\"{a}t Berlin. Also the author wants to thank the referee for his/her comments  which helped to improve the manuscript and to Carlos Cabrelli for carefully reading the paper.


\begin{thebibliography}{10}

\bibitem{AG01}
A.~Aldroubi and K.~Gr{\"o}chenig.
\newblock Nonuniform sampling and reconstruction in shift-invariant spaces.
\newblock {\em SIAM Rev.}, 43(4):585--620 (electronic), 2001.

\bibitem{ACRS05}
J.~Antezana, G.~Corach, M.~Ruiz, and D.~Stojanoff.
\newblock Nullspaces and frames.
\newblock {\em J. Math. Anal. Appl.}, 309(2):709--723, 2005.

\bibitem{BHP14a}
D.~Barbieri, E.~Hern{\'a}ndez, and J.~Parcet.
\newblock Riesz and frame systems generated by unitary actions of discrete
  groups.
\newblock {\em To appear in App. Comput. Harmon. Anal.,} DOI:
10.1016/j.acha.2014.09.007, 2014.

\bibitem{BHP14b}
D.~Barbieri, E.~Hern{\'a}ndez, and V.~Paternostro.
\newblock The Zak Transform and the structure of spaces invariant under the action of an LCA group. 
\newblock {\em Preprint}, 2014.

\bibitem{Bow00}
M.~Bownik.
\newblock The structure of shift-invariant subspaces of {$L^2({\Bbb R}^n)$}.
\newblock {\em J. Funct. Anal.}, 177(2):282--309, 2000.

\bibitem{BK06}
M.~Bownik and N.~Kaiblinger.
\newblock Minimal generator sets for finitely generated shift-invariant
  subspaces of {$L^2(\Bbb R^n)$}.
\newblock {\em J. Math. Anal. Appl.}, 313(1):342--352, 2006.

\bibitem{BR14}
M.~Bownik and K.~A. Ross.
\newblock The structure of translation-invarian spaces on locally compact
  abelian groups.
\newblock {\em Preprint}, 2014.

\bibitem{CMP14}
C.~Cabrelli, C.~Mosquera, and V.~Paternostro.
\newblock Linear combination of frame generators in systems of translates.
\newblock {\em J. Math. Anal. Appl.}, 413(2):776--788, 2014.

\bibitem{CP10}
C.~Cabrelli and V.~Paternostro.
\newblock Shift-invariant spaces on {LCA} groups.
\newblock {\em J. Funct. Anal.}, 258(6):2034--2059, 2010.

\bibitem{CP12}
C.~Cabrelli and V.~Paternostro.
\newblock Shift-modulation invariant spaces on {LCA} groups.
\newblock {\em Studia Math.}, 211(1):1--19, 2012.

\bibitem{dBDVR94a}
C.~de~Boor, R.~A. DeVore, and A.~Ron.
\newblock Approximation from shift-invariant subspaces of {$L_2(\Bbb R^d)$}.
\newblock {\em Trans. Amer. Math. Soc.}, 341(2):787--806, 1994.

\bibitem{dBDVR94b}
C.~de~Boor, R.~A. DeVore, and A.~Ron.
\newblock The structure of finitely generated shift-invariant spaces in
  {$L_2({\Bbb R}^d)$}.
\newblock {\em J. Funct. Anal.}, 119(1):37--78, 1994.

\bibitem{Deu95}
F.~Deutsch.
\newblock The angle between subspaces of a {H}ilbert space.
\newblock In {\em Approximation theory, wavelets and applications ({M}aratea,
  1994)}, volume 454 of {\em NATO Adv. Sci. Inst. Ser. C Math. Phys. Sci.},
  pages 107--130. Kluwer Acad. Publ., Dordrecht, 1995.

\bibitem{Gro04}
K.~Gr{\"o}chenig.
\newblock Localization of frames, {B}anach frames, and the invertibility of the
  frame operator.
\newblock {\em J. Fourier Anal. Appl.}, 10(2):105--132, 2004.

\bibitem{HJ95}
V.~P. Havin and B.~J{\"o}ricke.
\newblock The uncertainty principle in harmonic analysis [ {MR}1129019
  (93e:42001)].
\newblock In {\em Commutative harmonic analysis, {III}}, volume~72 of {\em
  Encyclopaedia Math. Sci.}, pages 177--259, 261--266. Springer, Berlin, 1995.

\bibitem{Hel64}
H.~Helson.
\newblock {\em Lectures on invariant subspaces}.
\newblock Academic Press, New York, 1964.

\bibitem{HSWW10}
E.~Hern{\'a}ndez, H.~{\v{S}}iki{\'c}, G.~Weiss, and E.~Wilson.
\newblock Cyclic subspaces for unitary representations of {LCA} groups;
  generalized {Z}ak transform.
\newblock {\em Colloq. Math.}, 118(1):313--332, 2010.

\bibitem{HW96}
E.~Hern{\'a}ndez and G.~Weiss.
\newblock {\em A first course on wavelets}.
\newblock Studies in Advanced Mathematics. CRC Press, Boca Raton, FL, 1996.
\newblock With a foreword by Yves Meyer.

\bibitem{KR10}
R.~A. Kamyabi~Gol and R.~Raisi~Tousi.
\newblock A range function approach to shift-invariant spaces on locally
  compact abelian groups.
\newblock {\em Int. J. Wavelets Multiresolut. Inf. Process.}, 8(1):49--59,
  2010.

\bibitem{Kat76}
T.~Kato.
\newblock {\em Perturbation theory for linear operators}.
\newblock Springer, New York, 1976.

\bibitem{Mac52}
G.~W. Mackey.
\newblock Induced representations of locally compact groups. {I}.
\newblock {\em Ann. of Math. (2)}, 55:101--139, 1952.

\bibitem{Mal89}
S.~G. Mallat.
\newblock Multiresolution approximations and wavelet orthonormal bases of
  {$L^2({\Bbb R})$}.
\newblock {\em Trans. Amer. Math. Soc.}, 315(1):69--87, 1989.

\bibitem{RS95}
A.~Ron and Z.~Shen.
\newblock Frames and stable bases for shift-invariant subspaces of {$L_2(\Bbb
  R^d)$}.
\newblock {\em Canad. J. Math.}, 47(5):1051--1094, 1995.

\bibitem{Rud62}
W.~Rudin.
\newblock {\em Fourier analysis on groups}.
\newblock Interscience Tracts in Pure and Applied Mathematics, No. 12.
  Interscience Publishers (a division of John Wiley and Sons), New York-London,
  1962.

\bibitem{Sri64}
T.~P. Srinivasan.
\newblock Doubly invariant subspaces.
\newblock {\em Pacific J. Math.}, 14:701--707, 1964.

\bibitem{Sun10}
Q.~Sun.
\newblock Local reconstruction for sampling in shift-invariant spaces.
\newblock {\em Adv. Comput. Math.}, 32(3):335--352, 2010.

\bibitem{Sun05}
W.~Sun.
\newblock Sampling theorems for multivariate shift invariant subspaces.
\newblock {\em Sampl. Theory Signal Image Process.}, 4(1):73--98, 2005.

\bibitem{ZZC06}
P.~Zhao, C.~Zhao, and P.~G. Casazza.
\newblock Perturbation of regular sampling in shift-invariant spaces for
  frames.
\newblock {\em IEEE Trans. Inform. Theory}, 52(10):4643--4648, 2006.

\end{thebibliography}

\end{document}